\documentclass{amsart}
\usepackage{amsfonts,amssymb,amscd,amsmath,enumerate,verbatim,calc}
\usepackage[all]{xy}

\newcommand{\CM}{Cohen-Macaulay}

\newcommand{\wrt}{with respect to}

\newcommand{\F}{\mathbb{F} }
\newcommand{\G}{\mathbb{G} }
\newcommand{\n}{\mathfrak{n} }
\newcommand{\m}{\mathfrak{m} }

\newcommand{\cf}{\mathbf{f} }

\newcommand{\Z}{\mathbb{Z} }

\newcommand{\rt}{\rightarrow}

\newcommand{\ov}{\overline}

\newcommand{\Bf}{\mathbf{f} }
\newcommand{\Bg}{\mathbf{g} }

\newcommand{\wt}{\widetilde }

\newcommand{\depth}{\operatorname{depth}}
\newcommand{\ord}{\operatorname{ord}}
\newcommand{\ann}{\operatorname{ann}}

\newcommand{\projdim}{\operatorname{projdim}}
\newcommand{\coker}{\operatorname{coker}}

\newcommand{\cx}{\operatorname{cx}}
\newcommand{\Syz}{\operatorname{Syz}}

\newcommand{\Ext}{\operatorname{Ext}}
\newcommand{\codim}{\operatorname{codim}}
\newcommand{\Tor}{\operatorname{Tor}}

\theoremstyle{plain}

\newtheorem{theorem}{Theorem}[section]

\theoremstyle{definition}

\newtheorem{s}[theorem]{}

\newtheorem{example}[theorem]{Example}

\theoremstyle{remark}

\begin{document}

\title[Resolutions]{Resolutions over strict complete intersections}
\author{Tony~J.~Puthenpurakal}
\date{\today}
\address{Department of Mathematics, IIT Bombay, Powai, Mumbai 400 076}

\email{tputhen@math.iitb.ac.in}

 \begin{abstract}
Let $(Q, \n)$ be a regular local ring and let $f_1, \ldots, f_c \in \n^2$ be a $Q$-regular sequence. Set $(A, \m) = (Q/(\cf), \n/(\cf))$. Further assume that the initial forms
 $f_1^*, \ldots, f_c^*$ form a $G(Q) = \bigoplus_{n \geq 0}\n^i/\n^{i+1}$-regular sequence.  Without loss of any generality assume
$\ord_Q(f_1) \geq \ord_Q(f_2) \geq \cdots \geq \ord_Q(f_c)$. Let $M$ be a finitely generated $A$-module and let $(\F, \partial)$ be a minimal free resolution of $M$. Then we prove that $\ord(\partial_i) \leq \ord_Q(f_1) - 1$ for all $i \gg 0$. We also construct an MCM $A$-module $M$ such that $\ord(\partial_{2i+1}) = \ord_Q(f_1) - 1$ for all $i \geq 0$.
We also give a considerably simpler proof regarding the periodicity of ideals of minors of maps in a  minimal free resolution of modules over arbitrary complete intersection rings (not necessarily strict).
\end{abstract}
 \maketitle
\section{introduction}
Study of resolutions is an important aspect of commutative algebra and algebraic geometry. Nevertheless many questions remain open. We will assume that the Noetherian local ring $(A, \m)$ is not regular. We investigate minimal free resolutions $\F$ of a finitely generated $A$-module $M$ with infinite projective dimension.
In the artice \cite{Ainf},  Avramov writes:
\emph{A striking aspect of infinite resolutions is the asymptotic
stability that they display, both numerically (uniform patterns of Betti numbers)
and structurally (high syzygies have similar module-theoretic properties). In many
cases this phenomenon can be traced to a simple principle: the beginning of a resolution
is strongly influenced by the defining relations of the module, that can be
arbitrarily complicated; far from the source, the relations of the ring (thought of
as a residue of some regular local ring) take over. In other words, the singularity
of the ring dominates asymptotic behavior.}

Perhaps the case when $A$ is a complete intersection is most studied. Let $M$ be a finitely generated $A$-module. Let $(\F, \partial)$ be a minimal free resolution of $M$. Let $I^r_i(M)$ be the ideal generated by $r \times r$ minors of $\partial_i$.  In the nice paper \cite[1.2]{bds} it is proved that $I^r_i(M) = I^r_{i+2}(M)$ for all $r \geq 1$ and $i \gg 0$. In particular $I^1_i(M) = I^1_{i+2}(M)$ for all $i \gg 0$. In this paper we investigate the order of this ideal when $A$ is a strict complete intersection.

Let $(A,\m)$ be a Noetherian local ring with residue field $k$. For $x \in \m^i \setminus \m^{i+1} $ set $\ord_A(x) = i$. Set $\ord_A(0) = + \infty$. If $I$ is a non-zero ideal in $A$ set
$\ord_A(I) = \max\{ i \mid I \subseteq \m^i \}$. Set $\ord_A  \{ 0 \} = + \infty$. Let $F, G$ be finite free $A$-modules and let $\phi \colon F \rt G$ be an $A$-linear map. Express $\phi$ as a matrix $(a_{ij})$. If $\phi \neq 0$ then set $\ord_A (\phi) = \max \{ l \mid a_{ij} \subseteq \m^l \ \text{for all} \ i,j \}$. We note that $\ord(\phi)$ only depends on $\phi$ and is independent of the bases of $F, G$ taken to express $\phi$ as a matrix.
If $\phi = 0$ set $\ord_A (\phi) = + \infty$. If $M$ is a finitely generated $A$-module and if $(\F, \partial)$ and $(\F^\prime, \partial^\prime)$ are two minimal resolutions of $M$ then note that
$\ord_A(\partial_i) = \ord_A(\partial^\prime_i)$ for all $i \geq 1$.

 Let $\mu(M)$ denote minimal number of generators of an $A$-module $M$ and let $\ell(M)$ denote its length. Let $\codim(A) = \mu(\m) - d$ denote the codimension of $A$.
Let $G(A) = \bigoplus_{n\geq 0}\m^n/\m^{n+1}$ be the associated graded ring of $A$ (\wrt \ $\m$) and let $G(M) = \bigoplus_{\n\geq 0}\m^nM/\m^{n+1}M$ be the associated graded module of $M$ considered as a $G(A)$-module.

\s \label{setup} Recall a local ring $A$ is said to be a strict complete intersection if the associated graded ring $G(A)$  of $A$ is a complete intersection. Note a strict complete intersection is a complete intersection but the converse need not hold. Assume  a strict complete intersection $A$ is a quotient of a regular local ring $(Q, \n)$. Then we can show that there exists a $Q$-regular sequence $f_1, \ldots, f_c$ such that $A = Q/(\cf)$ and the initial forms $f_1^*, \ldots, f_c^*$ is a $G(Q)$-regular sequence. In this case we have
$G(A) = G(Q)/(\cf^*)$. We may assume without loss of generality that $\ord_Q(f_1) \geq \ord_Q(f_2) \geq \cdots \geq \ord_Q(f_c)$. The main result of this paper is

\begin{theorem}\label{main}(with hypotheses as in \ref{setup}). Let $M$ be a finitely generated $A$-module and let $(\F, \partial)$ be a minimal free resolution of $M$ as an $A$-module. Then for all $i \gg 0$ we have $\ord_A(\partial_i) \leq \ord_Q(f_1) - 1$.
\end{theorem}
We complement Theorem \ref{main} by showing:
\begin{example}\label{ex}(with hypotheses as in \ref{setup}). There exists a finitely generated $A$-module $K$ such that if $(\F, \partial)$ is an minimal free resolution of $K$ as an $A$-module, then for all $i \geq 0$ we have $\ord_A(\partial_{2i+1}) = \ord_Q(f_1) - 1$.
\end{example}

In this paper we also give a considerably simpler proof of a result in \cite[1.2]{bds} of periodicity of ideal of minors of a resolution of a module $M$ over a complete intersection (not necessarily strict). We show
\begin{theorem}\label{minor}
Let $(A, \m)$ be a local complete intersection and let $M$ be a finitely generated $A$-module. Let $(\F, \partial)$ be a minimal free resolution of $M$. Let $I^r_i(M)$ be the ideal generated by $r \times r$ minors of $\partial_i$. Then  $I^r_i(M) = I^r_{i+2}(M)$ for all $r \geq 1$ and $i \gg 0$.
\end{theorem}

Here is an overview of the contents of this paper. In section two we discuss a few preliminaries that we need. In section three we discuss the notion of Eisenbud operators in a complete intersection and also some preliminaries on matrix factorizations. In section four we prove Theorem \ref{main} in the case when $M$ is MCM $A$-module with complexity one.
In the next section we discuss a well-known technique to reduce complexity of a module over complete intersections. In section six we prove Theorem \ref{main}. In the next section we give our example \ref{ex}. Finally in section eight we give a proof of Theorem \ref{minor}.

\section{Preliminaries}
In this paper all rings are Noetherian and all modules considered are assumed to be finitely generated.
In this section we discuss a few preliminary results that we need.
Let $(A,\m)$ be
 a local ring of dimension $d$ with residue field $k = A/\m$. Let $M$ be
 an
$A$-module. If $m$ is a non-zero
 element of $M$ and if $j$ is the largest integer such that $m \in \m^{j}M$,
then we let $m^*$ denote the image of $m$ in $\m^jM/\m^{j+1}M$.  The element $m^*$ of $G(M)$ is called the \emph{initial form} of $m$.
Also $\ell(N)$ denotes the length of an $A$-module $N$.
\s Let $H_M(z) = \sum_{n \geq 0} \ell_A(\m^nM/\m^{n+1}M)z^n$ be the Hilbert series of $M$. Then $H_M(z) = h_M(z)/(1-z)^{\dim M}$ where $h_M(z) \in \Z[z]$ and $h_M(1) \neq 0$. The polynomial $h_M(z)$ is called the \emph{$h$-polynomial} of $M$.

\s \textbf{Base change:}
\label{AtoA'}
 Let $\phi \colon (A,\m) \rt (A',\m')$ be a local ring homomorphism. Assume
   $A'$ is a faithfully flat $A$
algebra with $\m A' = \m'$. Set $\m' = \m A'$ and if
 $N$ is an $A$-module set $N' = N\otimes_A A'$.
 In these cases it can be seen that

\begin{enumerate}[\rm (1)]
\item
$\ell_A(N) = \ell_{A'}(N')$.
\item
 $H(M,n) = H(M',n)$ for all $n \geq 0$.
\item
$\dim M = \dim M'$ and  $\depth_A M = \depth_{A'} M'$.
\item
$h_M(z) = h_{M'}(z)$.
\item
$\projdim_A M = \projdim_{A'} M'$.
\item
$A'$ is a (strict) local complete intersection if and only if $A$ is a (strict) local complete intersection.
\item
Let $(\F.\partial)$ be a minimal resolution of $M$. Then $(\F', \partial') = (\F\otimes A', \partial \otimes A')$ is a minimal resolution of $M'$.
\item
Let $G, H$ be free $A$-modules and let $\phi \colon G \rt H$ be $A$-linear. Consider $\phi' \colon G' \rt F'$. Then
for every $r \geq 1$ we have $I^r(\phi') = I^r(\phi)A'$. Furthermore $\ord_A(\phi) = \ord_{A'}\phi'$.
\item
If $I, J$ are ideals in $A$. If $I' = J'$ then $I = J$.
\end{enumerate}

 \noindent The specific base changes we do are the following:

(i) $A' = A[X]_S$ where $S =  A[X]\setminus \m A[X]$.
The maximal ideal of $A'$ is $\n = \m A'$.
The residue
field of $A'$ is $K = k(X)$.

(ii) $A' = \widehat{A}$ the completion of $A$ with respect to the maximal ideal.

Thus we can assume that our ring $A$ is complete with infinite residue field.

\s Let $\beta_n(M) = \ell(\Tor^A_n(M, k))$ be the $n^{th}$-betti number of $M$. The complexity of an $A$-module is defined as
\[
\cx_A M = \inf \{ m \geq 0 \mid \limsup_{n \rt \infty} \beta_n(M)/n^{m-1} < \infty \}.
\]

\s \label{ord} Let $(Q,\n)$ be a regular local ring and let $f_1, \ldots, f_c \in \n$ be a regular sequence. Set $\ord_Q f_i = s_i$. Assume that $f_1^*, \ldots, f_c^*$ is a $G(Q)$-regular sequence. Set $A = Q/(f_1, \ldots, f_c)$ and $B = Q/(f_1, \ldots, f_{c -1})$.
Then $\ord_B f_c = s_c$.

To see this note that
$$G(B) = G(Q)/(f_1^*, \cdots, f_{c-1}^*) \quad \text{and} \quad G(A) = G(Q)/(f_1^*, \cdots, f_{c}^*) = G(B)/(f_c^*). $$
We note that the $h$-polynomial of $A$ and $B$ are as follows
$$h_A(z) = \prod_{i = 1}^{c}(1+z+ \cdots + z^{s_i -1}) \quad \text{and} \quad h_B(z) = \prod_{i = 1}^{c-1}(1+z+ \cdots + z^{s_i -1}).$$
Let $r = \ord_B f_c$. As $f_c^*$ is $G(B)$-regular and  $G(A) = G(B)/(f_c^*)$ it follows that
$$h_A(z) = h_B(z)(1+z+ \cdots + z^{r -1}).$$
Comparing the two expressions of $h_A(z)$ we get $r = s_c$.

\s \label{ini} Let $A = B/(f)$ . Let $b \in B$ with $\ord_B(b) \leq \ord_B(f) - 1$. Let $\ov{b}$ be the image of $b$ in $A$.
 Then $\ord_A(\ov{b}) \leq \ord_B(f) - 1$. To see this
 if  $\ord_A(\ov{b}) \geq \ord_B(f) $ then $b = b_1 + rf$ where $\ord_B(b_1) \geq \ord_B(f)$. It follows that $\ord_B(b) \geq \ord_B(f)$ which is a contradiction.
\section{Eisenbud operators and matrix factorization's}
To prove Theorem \ref{main} and \ref{minor}  we need the notion of cohomological operators over a complete intersection ring; see \cite{Gull} and
\cite{Eis}. We also need the theory of matrix factorizations.
 Let $\mathbf{f} = f_1,\ldots,f_c$ be a regular sequence in a  local Noetherian ring $Q$. Set $I = (\mathbf{f})$ and
 $ A = Q/I$,
\s
The \emph{Eisenbud operators}, \cite{Eis}  are constructed as follows: \\
Let $\mathbb{F} \colon \cdots \rightarrow F_{i+2} \xrightarrow{\partial} F_{i+1} \xrightarrow{\partial} F_i \rightarrow \cdots$ be a complex of free
$A$-modules.

\emph{Step 1:} Choose a sequence of free $Q$-modules $\wt{F}_i$ and maps $\wt{\partial}$ between them:
\[
\wt{\mathbb{F}} \colon \cdots \rightarrow \wt{F}_{i+2} \xrightarrow{\wt{\partial}} \wt{F}_{i+1} \xrightarrow{\wt{\partial}} \wt{F}_i \rightarrow \cdots
\]
so that $\mathbb{F} = A\otimes\wt{\mathbb{F}}$

\emph{Step 2:} Since $\wt{\partial}^2 \equiv 0 \ \text{modulo} \ (\mathbf{f})$, we may write  $\wt{\partial}^2  = \sum_{j= 1}^{c} f_j\wt{t}_j$ where
$\wt{t_j} \colon \wt{F}_i \rightarrow \wt{F}_{i-2}$ are linear maps for every $i$.

 \emph{Step 3:}
Define, for $j = 1,\ldots,c$ the map $t_j = t_j(Q, \mathbf{f},\mathbb{F}) \colon \mathbb{F} \rightarrow \mathbb{F}(-2)$ by $t_j = A\otimes\wt{t}_j$.

\s
The operators $t_1,\ldots,t_c$ are called Eisenbud's operator's (associated to $\mathbf{f}$) .  It can be shown that
\begin{enumerate}
\item
$t_i$ are uniquely determined up to homotopy.
\item
$t_i, t_j$ commute up to homotopy.
\end{enumerate}

\s \label{Basis-change} Let $\mathbf{g} = g_1,\ldots, g_c$ be another regular sequence in $Q$ and assume $I = (\Bf) = (\Bg)$. The Eisenbud operator's associated to
$\Bg$ can be constructed by using the corresponding operators associated to $\Bf$. An explicit construction is as follows:
Let
$$f_i = \alpha_{i1}g_1 + \cdots + \alpha_{ic} g_c \quad \text{for} \ i = 1, \ldots, c. $$
Then we can choose
\begin{equation}\label{bchange-eqn}
t_i^\prime = \alpha_{1i}t_1 + \cdots + \alpha_{ci} t_c \quad \text{for} \ i = 1, \ldots, c.
\end{equation}
as Eisenbud operators for $g_1,\ldots,g_c$. Note that in \cite{Eis} there is a mistake in indexing.

In the above setup
it is convenient to use matrices.
Here $\alpha = (\alpha_{ij})$ is a $c\times c$ invertible matrix with coefficients in $Q$. If $ \phi = (\phi_{ij})$ be a $m \times n$ matrix with coefficients in $Q$ then we set $\phi^{tr} = (\phi_{ji})$ to be the transpose of $\phi$. Set $[\Bf]$ and $[\Bg]$ to be the column vectors
$(f_1,\ldots,f_c)^{tr}$ and $(g_1,\ldots,g_c)^{tr}$ respectively. Thus in matrix terms we have
\[
[\Bf ] = \alpha\cdot [\Bg].
\]
Let $t_1, \ldots, t_c$ be the operators associated to $\Bf$. Set $[\mathbf{t}] = (t_1,\ldots,t_c)^{tr}$.`
Let $t_1^\prime,\cdots,t_c^\prime$ be our chosen Eisenbud operators associated to $\Bg$.
Then we have
\[
[\mathbf{t^\prime}] = \alpha^{tr}\cdot [\mathbf{t}].
\]

\s Let $R = A[t_1,\ldots,t_c]$ be a polynomial ring over $A$ with variables $t_1,\ldots,t_c$ of degree $2$. Let $M, N$ be  finitely generated $A$-modules. By considering a free resolution $\mathbb{F}$ of $M$ we get well defined maps
\[
t_j \colon \Ext^{n}_{A}(M,N) \rightarrow \Ext^{n+2}_{R}(M,N) \quad \ \text{for} \ 1 \leq j \leq c  \ \text{and all} \  n,
\]
which turn $\Ext_A^*(M,N) = \bigoplus_{i \geq 0} \Ext^i_A(M,N)$ into a module over $R$. Furthermore these structure depend  on $\Bf$, are natural in both module arguments and commute with the connecting maps induced by short exact sequences.

\s\label{avr}  Gulliksen, \cite[3.1]{Gull},  proved that if $\projdim_Q M$ is finite then
$\Ext_A^*(M,N) $ is a finitely generated $R$-module. For $N = k$, the residue field of $A$, Avramov in \cite[3.10]{LLAV} proved a converse; i.e., if
$\Ext_A^*(M,k)$ is a finitely generated $R$-module then $\projdim_Q M$ is finite. For a more general result, see \cite[4.2]{AGP}.

\s\label{r} Since $\m \subseteq \ann \Ext^{i}_A(M,k)$ for all $i \geq 0$ we get that $\Ext^*_A(M,k)$ is a module over $S = R/\m R = k[t_1,\ldots,t_c]$.
If $\projdim_Q M$ is finite then $\Ext^*_A(M,k)$ is a finitely generated $S$-module of Krull dimension $\cx M$.

\s \label{inverse}
Going mod $\m$ we get that the ring $S$ is invariant of a minimal generating set of $I = (\Bf)$.
Conversely if $\xi_1,\ldots,\xi_c \in S_2$ be such that $S = k[\xi_1,\ldots,\xi_c]$ then there exist's  a regular sequence
$\Bg = g_1,\ldots,g_c$ such that
\begin{enumerate}
\item
$(\Bg) = (\Bf)$
\item
if $t_j^\prime$ are the Eisenbud operators associated to $g_j$ for $j = 1,\ldots,c$ then the action of $t_j^\prime$ on $\Ext^*_A(M,k)$ is same as that
of $\xi_j$ for $j = 1,\ldots,c$.
\end{enumerate}
This can be seen as follows.
Let
\[
\xi_i = \ov{\beta_{i1}}t_1 + \cdots + \ov{\beta_{ic}}t_c \quad \text{for} \ i = 1,\ldots,c.
\]
Let $\beta = (\beta_{ij}) \in M_c(A)$. Note that $\beta$ is an invertible matrix since $\ov{\beta} = (\ov{\beta_{ij}})$ is an invertible matrix in $M_n(k)$. Set $\alpha = \beta^{tr}$. Define $\Bg = g_1,\ldots,g_c$ by
\[
[\Bg] = \alpha^{-1}\cdot[\Bf]
\]
Clearly $\Bg = g_1,\ldots,g_c$ is a regular sequence and $(\Bf) = (\Bg)$.
Notice $[\Bf] = \alpha \cdot [\Bg]$. So by \ref{Basis-change} the cohomological operators $t_1^{\prime},\ldots, t_c^{\prime}$ associated to $\Bg$
is given by the formula
\[
[\mathbf{t^\prime}] = \alpha^{tr}\cdot [\mathbf{t}] = \beta \cdot [\mathbf{t}].
\]
It follows that the action of $t_j^\prime$ on $\Ext^*_A(M,k)$ is same as that
of $\xi_j$ for $j = 1,\ldots,c$.

\s\label{mf}(\cite[5.1, 5.2]{Eis}) A \emph{matrix factorization} of an element $f$ in a ring $P$ is an ordered pair of maps of free $P$ -modules ($\phi \colon F\rt G$ and $\psi \colon G \rt F$) such that
 $\psi\phi = f1_F$ and $\phi\psi = f1_G$. Note that if $(\phi, \psi)$ is a matrix factorization of $f$, then $f$ annihilates $\coker \phi$.  Set $M = \coker \phi$.
If $f$ is $P$-regular then we have the following periodic   resolution of $M$ over $A = P/(f)$,
$$  \cdots \rt \ov{G} \xrightarrow{\ov{\phi}} \ov{H} \xrightarrow{\ov{\psi}} \ov{G}\rt \cdots \rt\ov{G} \xrightarrow{\ov{\phi}} \ov{H} \rt M \rt 0. $$
We note that this is a \textit{minimal} resolution of $M$ if $M$ has \textit{ no} free summands as an $A$-module.

Conversely if $f$ is a non-zero divisor of $P$ and $M$ is an $A = P/(f)$ module with $\projdim_P M = 1$ then there exists a matrix factorization $(\phi, \psi)$ of $f$ such that $\coker \phi = M$, see \cite[5.1.2]{Ainf}.

\section{The case when $M$ has complexity one}
In this section we give a proof of Theorem \ref{main} when $M$ is a maximal \CM \ $A$-module of complexity one over a strict complete intersection. Throughout the set-up will
be as in \ref{setup}.
\begin{theorem}
\label{cx1}Let $M$ be a MCM $A$-module of complexity one. Let $(\F, \partial)$ be a minimal resolution of $M$. Then $\ord_A(\partial_i) \leq \ord_Q(f_1) - 1$.
\end{theorem}
\begin{proof}
We may assume that $M$ has no free summands. We also assume that the residue field $k$ of $A$ is infinite.
 We note that $E = \Ext^*_A(M,k)$ is a finitely generated graded $S = k[t_1,\ldots,t_c]$ module of Krull dimension one, see \ref{r}. It is easy to see that
there exists $\xi_1 = \ov{\beta_1}t_1 + \cdots + \ov{\beta_c}t_c \in S_2$ such that
\begin{enumerate}
\item
$\xi_1$ is a parameter for $E$.
\item
$\ov{\beta_1}, \ov{\beta_2}, \ldots, \ov{\beta_c}$ are all non-zero.
\end{enumerate}
Set $\xi_j = t_j$ for $j = 2,\ldots,c$. Then $S = k[\xi_1, \ldots, \xi_c]$.

Set $\beta = (\beta_{ij})$ where
\[
\beta_{ij}  =
\begin{cases}
\beta_j, &\text{if $i = 1$;} \\
1, &\text{if $i = j$ and $i \geq 2$;} \\
0, &\text{otherwise.}
\end{cases}
\]
Clearly $\beta$ is an invertible matrix in $M_n(A)$. By \ref{inverse} there exists a regular sequence $\Bg = g_1, \ldots,g_c$ such that
\begin{enumerate}
\item
$(\Bg) = (\Bf)$
\item
if $t_j^\prime$ are the Eisenbud operators associated to $g_j$ for $j = 1,\ldots,c$ then the action of $t_j^\prime$ on $\Ext^*_A(M,k)$ is same as that
of $\xi_j$ for $j = 1,\ldots,c$.
\end{enumerate}
By \ref{inverse}
\[
[\Bg] = (\beta^{tr})^{-1}\cdot [\Bf] = (\beta^{-1})^{tr} \cdot [\Bf].
\]
It is easy to compute the inverse of $\beta$. So we obtain
\begin{align*}
g_1 &= \frac{1}{\beta_1}f_1 \\
g_j &= \frac{-\beta_j}{\beta_1} f_1 + f_j \ \text{for} \ j = 2,\ldots,c.
\end{align*}
Recall that $\ord(f_1) \geq \ord(f_j)$ for $j = 2, \ldots,c$. Notice that if $\ord(f_1) = \ord(f_j)$ then
\[
\frac{-\beta_j}{\beta_1} f_1^* + f_j^* \neq 0,
\]
since $f_1^*, \ldots, f_c^*$ is a $G(Q)$-regular sequence. Thus we have
\[
g_1^* = \frac{1}{\beta_1}f_1^*  \quad \text{and}
\]
for $2 \leq j \leq c$ we have
\[
g_j^* =
\begin{cases}
f_j^*, &\text{if $\ord(f_1) > \ord(f_j)$;} \\
\frac{-\beta_j}{\beta_1} f_1^* + f_j^*, &\text{if $\ord(f_1) = \ord (f_j)$.}
\end{cases}
\]
Notice that $\Bg^* = g_1^*,\ldots,g_c^*$ is a regular sequence in $G(Q)$ and $(\Bg^*) = (\Bf^*)$. The subring
$k[\xi_1]$ of $S$ can be identified with the ring $S^\prime$ of cohomological operators of a presentation $A = P/(g_1)$ where
$P = Q/(g_2,\ldots,g_c)$. Since $\Ext_A^*(M,k)$ is a finite module over $S^\prime$, by \ref{avr} we get that $\projdim_P M$ is finite.

Notice that
\begin{enumerate}
\item
$\dim P = \dim A + 1 \geq 2$.
\item
$\projdim_P M = 1$, (as $M$ is an MCM $A$-module).
\end{enumerate}
By \ref{mf} there is a matrix-factorization $(\phi, \psi)$ of $f_1$ over $P$ such that
$0 \rt G \xrightarrow{\phi} H \rt M \rt 0$ is a minimal free resolution of $M$ over $P$ and $\psi\phi = f_11_G $ and $\phi\psi = f_1 1_H$.
As $M$ has no free summands as an $A$-module it follows that $\psi$ is also a minimal map.
Furthermore we have the following  periodic minimal resolution of $M$ over $A$,
$$  \cdots \rt \ov{G} \xrightarrow{\ov{\phi}} \ov{H} \xrightarrow{\ov{\psi}} \ov{G}\rt \cdots \rt\ov{G} \xrightarrow{\ov{\phi}} \ov{H} \rt M \rt 0. $$
We note that $\ord_P(\phi), \ord_P(\psi) \leq \ord_P(f_1) - 1$. By \ref{ord},  $\ord_P(f_1) = \ord_Q(f_1)$. We also note that if $a \in P$ has order $\leq \ord_P(f_1) -1$ then its order in $A$ is also $\leq \ord_P(f_1) -1$, see \ref{ini}. It follows that
$\ord_A(\ov{\phi}), \ord_A(\ov{\psi}) \leq \ord_P(f_1) - 1$. The result follows.
\end{proof}
\section{Section of a module}
In this section we describe a well-known construction to reduce the complexity of a module over a complete intersection.
\begin{theorem}
\label{section}
Let $(Q,\n)$ be a local ring with infinite residue field and let $f_1, \ldots, f_c$ be a $Q$-regular sequence. Set $A= Q/(\cf)$. Let $M$ be an $A$-module with $\projdim_Q M$ finite. Assume $\cx_A M \geq 2$. Let $(\F, \partial)$ be a minimal resolution of $M$ as an $A$-module and let $t_1,\ldots,t_c \colon \F(+2) \rt \F $ be Eisenbud operators corresponding to $\Bf$ and $\F$. Set $M_n = \Syz^A_n(M)$. Set  $R = A[t_1,\ldots, t_c]$ (with $\deg t_i =2$ for all $i$). Then there
exists $\xi \in R_2$ such that
\begin{enumerate}[\rm (1)]
\item There exists $n_0$ such that for all $n\geq n_0$ the map of complexes $\xi \colon \F(+2) \rt \F$ induces surjections $\F_{n+2} \rt \F_n$. Let $G_n = \ker(\F_{n+2} \rt \F_{n})$.
\item Note $\xi$ induces surjections $M_{n+2} \rt M_n$ for $n \geq n_0$. Let $L_n = $ kernel of this map. Then $\cx_A L_{n_0} = \cx M -1$.
\item $\G[n_0] \rt L_{n_0}$ is a minimal resolution of $L_{n_0}$. Furthermore $L_n = \Syz^A_{n -n_0}(L_{n_0})$ for all $n \geq n_0$.
\item We have a commutative diagram for all $n \geq n_0 + 1$,
\[
  \xymatrix
{
 0
 \ar@{->}[r]
  & \G_n
\ar@{->}[r]
\ar@{->}[d]^{\delta_n}
 & \F_{n+2}
\ar@{->}[r]
\ar@{->}[d]^{\partial_{n+2}}
& \F_n
\ar@{->}[r]
\ar@{->}[d]^{\partial_{n} }
&0
\\
 0
 \ar@{->}[r]
  & \G_{n-1}
\ar@{->}[r]
 & \F_{n+1}
\ar@{->}[r]
& \F_{n-1}
    \ar@{->}[r]
    &0
\
 }
\]
\end{enumerate}
\item
Fix $n \geq n_0 +1$. There exists basis of $\F_m, \F_{m+2}$ and $\G_m$ (for $m = n, n-1$) such that  the matrix of $\partial_{n+2}$ is
\[
[\partial_{n+2}] =
\begin{pmatrix}
[\delta_n] & U_n \\ 0 & [\partial_n]
\end{pmatrix}.
\]
\end{theorem}
\begin{proof}
 (1) and (2): Set $E(M) = \bigoplus_{n \geq 0}\Ext^n_A(M, k)$. It is a finitely generated graded $R$-module where $\deg t_i = 2$ for all $i$.
As $\m E(M) = 0$ we get that $E(M)$ is a finitely generated $S = k[t_1,\ldots, t_c]$-module. As $k$ is infinite  we may choose $t$ homogeneous of degree $2 $ in $S$ such that $(0 \colon_{E(M)} t)$ has finite length (equivalently $t$ is $E(M)$-filter regular).
We have that $t$ induces injections $\Ext^{n}_A(M, k) \rt \Ext^{n+ 2}_A(M, k)$ for $n \gg 0$ say for $n \geq n_0$.
Dualizing we get surjections $\Tor^A_{n+2}(M, k) \rt \Tor^A_{n}(M, k)$ for $n \geq n_0$.
 Let $\xi$ be homogeneous of degree $2$ in $R$ such that its image in $S$ is $t$. So we have a chain map $\xi \colon \F(+2) \rt \F$. As $\Tor^A_{n}(\xi, k)$ is surjective for $n \geq n_0$, it follows from Nakayama's lemma  that we have surjections $\F_{n + 2} \rt \F_n$ for $n \geq n_0$. Note $\xi$ is a chain map. So we have exact sequences $ 0 \rt L_n \rt M_{n+ 2} \rt M_n \rt 0$ for $n \geq n_0$. Also note that we also have exact sequences  for $i \geq 0$
\[
\Tor^A_{i}(M_{n_0 + 2}, k) \rt \Tor^A_{i}(M_{n_0}, k) \rt 0.
\]
 It follows that we have exact sequence
\[
 0 \rt \Tor^A_{i}(L_{n_0}, k) \rt \Tor^A_{i}(M_{n_0 + 2}, k) \rt \Tor^A_{i}(M_{n_0}, k) \rt 0.
\]
So $\cx L_{n_0} = \cx M - 1$.

(3) As $\xi \colon \F[+2] \rt\F$ is a chain map we get that $\ker \xi$ is also a chain complex. It is easy to see that $H^*(\G[n_0]) = L_{n_0}$. So $\G[n_0] \rt L_{n_0}$ is a  resolution of $L_{n_0}$. To see it is minimal, let $n \geq n_0 + 1$; then $L_n \subseteq M_{n+2} \subseteq \m \F_{n+1}$. As the sequence $0 \rt \G_{n-1} \rt \F_{n+1} \rt \F_{n-1} \rt 0$ is split  we get $L_n \subseteq \m \F_{n+1} \cap G_{n-1} = \m \G_{n-1}$. So $\G[n_0] \rt L_{n_0}$ is a  minimal  resolution of $L_{n_0}$. It is elementary to see that  $L_n = \Syz^A_{n -n_0}(L_{n_0})$ for all $n \geq n_0$.

(4) The commutative diagram holds since $\xi$ is a chain map. Let \\ $C_m = \{c_1, \ldots, c_{s_m}\}$ and $V_m = \{v_1, \ldots, v_{l_m}\}$  be basis of $\G_m$ and $\F_m$ for $m =n, n-1$. Then
choose a basis of $\F_{m+2}$ (for $m = n-1, n$) as  $ \{c_1, \ldots, c_{s_m}, v_1^\prime , \ldots, v^\prime_{l_m}\}$ where $\xi$ maps $v_i^\prime$ to $v_i$. Clearly the matrix of $\partial_{n+2}$ with respect to this basis is of the required form.
\end{proof}

\section{Proof of Theorem \ref{main}}
In this section we give
\begin{proof}
[Proof of Theorem \ref{main}] We may assume $M$ is a maximal \CM  \\ $A$-module. We also assume the residue field of $A$ is infinite. We give a proof by induction on $\cx_A M$. When $\cx_A M = 1$  the result follows from \ref{cx1}. We now assume $r = \cx_A M  \geq 2$ and the result is proved for modules with complexity $\leq r -1$. We do the construction as in Theorem \ref{section}. Let $(\F, \partial)$ be a minimal resolution of
$M$. Then there exists $n \geq n_0$ such that we have exact sequence $0 \rt L_n \rt M_{n+2} \rt M_n \rt 0$ for all $n \geq n_0$. We have $\cx L_{n_0} = r -1$ and for all $n \geq n_0$ we have $L_n = \Syz^A_{n-n_0}(L_{n_0})$.
We also have that $(\G[n_0], \delta)$ is a minimal resolution of $L_{n_0}$. So by our  induction hypothesis we have $\ord_A(\delta_n) \leq \ord_Q(f_1) - 1$ for all $n \gg 0$, say $n \geq m_0$.

 Furthermore for $n \geq n_0 +1$ we have that after  appropriately choosing basis we have
  that  the matrix of $\partial_{n+2}$ is
\[
[\partial_{n+2}] =
\begin{pmatrix}
[\delta_n] & U_n \\ 0 & [\partial_n]
\end{pmatrix}.
\]
It follows that for $n > \max\{n_0 + 1, m_0\}$ we have $\ord_A(\partial_{n + 2}) \leq \ord_Q(f_1) - 1$. The result follows.
\end{proof}

\section{Construction of Example \ref{ex}}

In this section we give construction of Example \ref{ex}.

\s \emph{Construction:} Let $(Q,\n)$ be regular local. Assume that $\cf = f_1, \ldots f_c \in \n^2$ is a $Q$ regular sequence and $\cf^* = f_1^*,\ldots, f_c^*$ is a $G(Q)$-regular sequence. Assume $\ord_Q(f_1) \geq \ord_Q(f_2) \geq \cdots \geq \ord_Q(f_c)$.  Set
$A = Q/(\cf)$. We may assume $\ord_Q(f_1) \geq 2$ otherwise $A$ is regular ring (and we have nothing to show in this case).

We note that the hypersurface $B = Q/(f_1)$ has an Ulrich module $U$, \cite[Theorem, p.\ 189]{HBU}. Set $N = \Syz^B_1(U)$. Let $0 \rt G \xrightarrow{\phi} H \rt N \rt 0$ be a minimal presentation of $N$ over $Q$. Furthermore $N$ and $\phi$ induces a matrix factorization $(\phi, \psi)$ of $f$.
 Then by \cite[4.10]{P2},  we have $\ord_Q(\phi) = \ord_Q(f_1) - 1$.

 Set $M = N/(f_2,\ldots, f_c)N$ and $R = Q/(f_2, \ldots, f_c)$. Then
$\projdim_R M = 1$. Then $(\ov{\phi}, \ov{\psi})$ induces a matrix factorization of $f_1$ over $R$. It follows that \\  $\ord_R(\ov{\phi}), \ord_R({\psi}) \leq \ord_R(f_1) - 1 = \ord_Q(f_1) - 1$.
As $\ord_Q(\phi) = \ord_Q(f_1) - 1$, it follows that $\ord_R(\ov{\phi}) = \ord_R(f_1)- 1$. Furthermore we have the following  periodic minimal resolution of $M$ over $A$,
$$  \cdots \rt \ov{G} \xrightarrow{\ov{\phi}} \ov{H} \xrightarrow{\ov{\psi}} \ov{G}\rt \cdots \rt\ov{G} \xrightarrow{\ov{\phi}} \ov{H} \rt M \rt 0. $$
  We also note that if $a \in P$ has order $\leq \ord_P(f_1) -1$ then its order in $A$ is also $\leq \ord_P(f_1) -1$, see \ref{ini}. It follows that
$\ord_A(\ov{\phi}) =  \ord_P(f_1) - 1 = \ord_Q(f_1) - 1$. The result follows.

\section{Proof of Theorem \ref{minor}}
In this section we give
\begin{proof}[Proof of Theorem \ref{minor}] We may assume that the residue field of $A$ is infinite. We may also assume that $\projdim M = \infty$ otherwise the result hold vacuously.
We may also assume that $M$ is a MCM $A$-module with no free-summands.
We induct on $\cx M$. If $\cx M = 1$ then  $M$ is periodic with period $\leq 2$. So there is nothing to show.
We now assume that $\cx_A M \geq 2$. Fix $r \geq 1$. Let $(\F,\partial)$ be a minimal resolution of $M$. By \ref{section} we have that for $n \gg 0$, after choosing appropriate basis, the matrix of $\partial_{n+2}$ is given as
\[
[\partial_{n+2}] =
\begin{pmatrix}
[\delta_n] & U_n \\ 0 & [\partial_n]
\end{pmatrix}.
\]
So $I^r_n(M) \subseteq I^r_{n+2}(M)$. Thus for $n \gg 0$ we have a chain
\[
I^r_n(M) \subseteq I^r_{n+2}(M)  \subseteq I^r_{n +4}(M) \subseteq I^r_{n+6}(M) \subseteq \cdots
\]
As $A$ is Noetherian this chain stabilizes. The result follows.
\end{proof}

\end{document}